\documentclass{amsart}
\usepackage{amsmath}
\usepackage{amssymb}
\usepackage{amscd}
\usepackage{color}
\usepackage[all]{xy}

\newcommand{\GD}{\operatorname{GD}}
\newcommand{\gGD}{\operatorname{gGD}}
\newcommand{\gLO}{\operatorname{gLO}}
\newcommand{\Max}{\operatorname{Max}}
\newcommand{\Spec}{\operatorname{Spec}}
\newtheorem{theorem}{Theorem}[section]
\newtheorem{lemma}[theorem]{Lemma}
\newtheorem{corollary}[theorem]{Corollary}
\newtheorem{proposition}[theorem]{Proposition}
\newtheorem{definition}[theorem]{Definition}
\newtheorem{example}[theorem]{Example}
\newtheorem{remark}[theorem]{Remark}

\begin{document}

\title[On graded Going down domains]
{On graded Going down domains}
\author[P. Sahandi and N. Shirmohammadi]
{Parviz Sahandi and Nematollah Shirmohammadi}

\address{(Sahandi) Department of Pure Mathematics, Faculty of Mathematics, Statistics and Computer Science,  University of Tabriz, Tabriz, Iran.} \email{sahandi@ipm.ir}
\address{(Shirmohammadi) Department of Pure Mathematics, Faculty of Mathematics, Statistics and Computer Science,  University of Tabriz, Tabriz, Iran.} \email{shirmohammadi@tabrizu.ac.ir}

%\date{\today}


\begin{abstract}
Let  $\Gamma$ be a torsionless commutative cancellative monoid and $R =\bigoplus_{\alpha \in \Gamma}R_{\alpha}$ be a $\Gamma$-graded
integral domain. In this paper, we introduce the notion of graded going-down domains. Among other things, we provide an equivalent condition for graded-Pr\"{u}fer domains in terms of graded going-down and graded finite-conductor domains. We also characterize graded going-down domains by means of graded divided domains. As an application, we show that the graded going-down property is stable under factor domains.
\end{abstract}

\maketitle

\section{Introduction}

All rings considered in this paper are commutative rings with identity. An overring of a ring $R$ is a subring of the total quotient ring of $R$ containing $R$. In literature, an extension of commutative integral domains $R\subseteq S$ is said to satisfy the \emph{going-down property} ($\GD$ for short) provided that given primes $P_0\subseteq P$ in $R$ and $Q$ in $S$ with $Q\cap R=P$, there exists a prime ideal $Q_0$ of $S$ such that $Q_0\subseteq Q$ and $Q_0\cap R=P_0$ \cite[Page 28]{k74}. In \cite{d73, d74, d76}, Dobbs defined a commutative integral domain $R$ to be a \emph{going-down domain ($\GD$ domain)} in case $R\subseteq S$ satisfies going-down property for each domain $S$ containing $R$. In \cite[Theorem 1]{dp76}, Dobbs and Papick showed that one may restrict the test domains $S$ either to be valuation overrings of $R$ or to be simple overrings of $R$. Using this concept, he obtained a nice characterization of Pr\"{u}fer domains. In fact, he proved that $R$ is a Pr\"{u}fer domain if and only if $R$ is an integrally closed finite-conductor $\GD$ domain \cite[Corollary 4]{d73}. It is also shown that a quasi-local domain $R$ is a $\GD$ domain if and only if $R$ has a divided integral overring \cite[Theorem 2.5]{d76}.

The purpose of this paper is to study the concept of $\GD$ domain in case of a graded integral domain $R=\bigoplus_{\alpha\in\Gamma}R_{\alpha}$ graded by an arbitrary grading torsionless monoid $\Gamma$. In Section 2, we define the graded going-down property ($\gGD$ for short) for extensions of graded domains and show that the $\gGD$ property holds for an extension of graded domains whenever the contraction map of homogeneous prime spectra is open. In Section 3, we first give our definition of graded going-down domains ($\gGD$ domains). Then it is shown that a $\gGD$ domain has a treed homogeneous prime spectrum. This leads us to obtain a new characterization of graded-Pr\"{u}fer domains. We next show that in the definition of $\gGD$ domains we can only consider the test overrings either to be gr-valuation overrings or to be homogeneous simple overrings. We finally provide a characterization of $\gGD$ domains in the gr-Noetherian case. In Section 4, we show that the $\gGD$ property is stable under the formation of rings of fractions and factor domains. In order to achieve this, we need to find an equivalent condition for $\gGD$ domains by means of graded divided overrings.

\subsection{Graded integral domains}

Let $\Gamma$ be a (nonzero) torsionless commutative cancellative monoid (written
additively) and $\langle \Gamma \rangle = \{a - b \mid a,b \in \Gamma\}$ be the
quotient group of $\Gamma$; so $\langle \Gamma \rangle$ is a torsionfree abelian group.
It is well known that a cancellative monoid $\Gamma$ is torsionless
if and only if $\Gamma$ can be given a total order compatible with the monoid
operation \cite[page 123]{no68}. By a $(\Gamma$-)graded integral domain  $R =\bigoplus_{\alpha \in \Gamma}R_{\alpha}$,
we mean an integral domain graded by $\Gamma$.
That is, each nonzero $x \in R_{\alpha}$ has degree $\alpha$, i.e., deg$(x) = \alpha$,  and
deg$(0) = 0$. Thus, each nonzero $f \in R$ can be written uniquely as $f = x_{\alpha_1} + \dots + x_{\alpha_n}$ with
deg$(x_{\alpha_i}) = \alpha_i$ and $\alpha_1 < \cdots < \alpha_n$.
A nonzero $x \in R_{\alpha}$ for every $\alpha \in \Gamma$ is said
to be {\em homogeneous}.

The $\Gamma$-graded domain $R=\bigoplus_{\alpha\in\Gamma}R_{\alpha}$ is called a \emph{graded subring} of a $\Lambda$-graded domain $T =\bigoplus_{\alpha \in \Lambda}T_{\alpha}$, if $\Gamma$ is a subsemigroup of $\Lambda$ and for every $\alpha\in \Gamma$, $R_{\alpha}\subseteq T_{\alpha}$. It is easy to see that we can consider $R$ as a $\Lambda$-graded ring by setting $R_{\alpha}=0$ for each $\alpha\in\Lambda\setminus\Gamma$, and in this case $R=\bigoplus_{\alpha \in \Lambda}(T_{\alpha}\cap R)$. By ``$R\subseteq T$ is an extension of graded domains'' we mean that $R$ is a graded subring of $T$.

Let $R=\bigoplus_{\alpha\in\Gamma}R_{\alpha}$ be a $\Gamma$-graded domain and set  $H = \bigcup_{\alpha \in \Gamma}(R_{\alpha} \setminus \{0\})$; so
$H$ is the saturated multiplicative set of nonzero homogeneous elements of $R$.
Then $R_H$, called the {\em homogeneous quotient field} of $R$,
 is a $\langle \Gamma \rangle$-graded integral domain whose nonzero homogeneous elements are units.
We say that an overring $T$ of $R$ is a {\em homogeneous overring} of $R$ if
$T = \bigoplus_{\alpha \in \langle \Gamma \rangle}(T \cap (R_H)_{\alpha})$;
so $T$ is a $\langle \Gamma \rangle$-graded integral domain
such that $R \subseteq T \subseteq R_H$.
The ring of fractions $R_S$ is a homogeneous overring of $R$ for
a multiplicative set $S$ of nonzero homogeneous elements of $R$
(with deg$(\frac{a}{b}) =$ deg$(a) -$ deg$(b)$ for $a \in H$ and $b \in S$).

For a fractional ideal $A$ of $R =\bigoplus_{\alpha \in \Gamma}R_{\alpha}$ with $A \subseteq R_H$,
let $A^*$ be the fractional ideal of $R$ generated by homogeneous elements in $A$; so
$A^* \subseteq A$. The fractional ideal $A$ is said to be {\em homogeneous} if $A^* = A$. A homogeneous ideal $P$ of $R$ is prime if and only if for all $a, b\in H$, $ab\in P$ implies $a\in P$ or $b\in P$ \cite[Page 124]{no68}. The set of all homogeneous prime ideals of $R$ is denoted by $\Spec_h(R)$. We note that a minimal prime ideal $P$ of a homogeneous ideal $I$ of $R$ is homogeneous. Indeed, the inclusions $I\subseteq P^*\subseteq P$ shows that $P^*=P$.

A homogeneous ideal of $R$ is called a {\em maximal homogeneous ideal}
if it is maximal among proper homogeneous ideals of $R$.
It is easy to see that each proper homogeneous ideal of $R$
is contained in a maximal homogeneous ideal of $R$. The set of all maximal homogeneous ideals of $R$ is denoted by $\Max_h(R)$.

Throughout this paper, let $\Gamma$ be a nonzero torsionless commutative cancellative monoid, $R=\bigoplus_{\alpha\in\Gamma}R_{\alpha}$ be an integral domain graded by $\Gamma$ and $H$ be the set of nonzero homogeneous elements of $R$.

\section{Graded going down property}

In this section we define the graded theoretic version of going-down property and then identify some sufficient conditions for this property.

\begin{definition} Let $R\subseteq T$ be an extension of graded domains. We say that $R\subseteq T$ satisfies graded going-down (abbreviated, $\gGD$) if, whenever $P_0\subseteq P$ are homogeneous prime ideals of $R$ and $Q$ is a homogeneous prime ideal of $T$ such that $Q\cap R=P$, there exists a homogeneous prime ideal $Q_0$ of $T$ such that $Q_0\subseteq Q$ and $Q_0\cap R=P_0$.
\end{definition}

Assume that the extension of graded domains $R\subseteq T$ satisfies $\GD$. We show that $R\subseteq T$ satisfies $\gGD$. Indeed, let $P_0\subseteq P$ be homogeneous prime ideals of $R$ and $Q$ be a homogeneous prime ideal of $T$ such that $Q\cap R=P$. Since $R\subseteq T$ satisfies $\GD$, there exists a prime ideal $Q_0$ of $T$ such that $Q_0\subseteq Q$ and $Q_0\cap R=P_0$. Now using \cite[Lemma 2.7]{hs18} we have $P_0=P_0^*=(Q_0\cap R)^*=Q_0^*\cap R$. Thus $R\subseteq T$ satisfies $\gGD$. By Example \ref{e}, one can find an extension of graded domains satisfying $\gGD$ but not $\GD$.

Let $R\subseteq T$ be an extension of graded domains. Suppose that $R$ is integrally closed and that $T$ is integral over $R$. Then by \cite[Theorem 5.16]{AM}, $R\subseteq T$ satisfies (g)$\GD$.

Let $A \subseteq B$ be a unital extension of commutative rings. It is known that if the canonical contraction map $\Spec(B) \to \Spec(A)$ is open (relative to the Zariski topology), then $A\subseteq B$ satisfies $\GD$. (For a particularly accessible proof
of this, see \cite{DP2}. In fact, much more is true: see \cite[Corollaire 3.9.4 (i), page 254]{EGA}.) Assume that $R\subseteq T$ is an extension of graded domains. It is well-known that the contraction map on prime spectra restricts to a well defined function
$$
F:\Spec_h(T)\to \Spec_h(R),
\, \, Q \mapsto Q \cap R
$$
of topological spaces which is continuous (with respect to the
subspace topology induced by the Zariski topology).

\begin{proposition} %\label{theorem 1}
Let $R\subseteq T$ be an extension of graded domains. Suppose that the canonical continuous map $F:\Spec_h(T)\to\Spec_h(R)$ is open (with respect to Zariski subspace topologies). Then $R\subseteq T$ satisfies $\gGD$.
\end{proposition}
\begin{proof}
Assume $R\subseteq T$ does not satisfy $\gGD$. Then there exist homogeneous prime ideals $P\subsetneq P_1$ of $R$ and a homogeneous prime ideal $Q_1$ in $T$ such that $Q_1\cap R=P_1$ and no $Q\in\Spec_h(T)$ satisfies both $Q\subseteq Q_1$ and $Q\cap R=P$. Note that if there exists a prime ideal $Q'$ of $T$ satisfying both $Q'\subseteq Q_1$ and $Q'\cap R=P$, then using \cite[Lemma 2.7]{hs18} we obtain that $Q'^*\subseteq Q_1$ and $Q'^*\cap R=P$ which is absurd. Therefore $R\subseteq T$ does not satisfy GD. Applying \cite[Exercise 37, p. 44]{k74}, we find $Q\in\Spec(T)$ such that $Q$ is minimal over $PT$, $Q\subseteq Q_1$ and $$PT\cap(R\setminus P)(T\setminus Q)\neq\emptyset.$$ Hence $Q=Q^*\in\Spec_h(T)$. The above display leads to an equation $\sum_{i=1}^np_it_i=dt$, for some elements $p_i\in P$, $t_i\in T$, $d\in R\setminus P$ and $t\in T\setminus Q$. Set $X:=\Spec(T)$ and let $X_t$ denote the set of elements of $X$ that do not contain $t$. Note that $X_t$ is a basic Zariski open subset of $X$. Then $X_t\cap\Spec_h(T)$ is an open subset of $\Spec_h(T)$. By assumption $F(X_t\cap\Spec_h(T))$ is an open subset of $\Spec_h(R)$. It follows that $P\in F(X_t\cap\Spec_h(T))$ since $Q\in X_t\cap\Spec_h(T)$ and $P\subseteq Q\cap R$. Hence there exists $Q_0\in X_t\cap\Spec_h(T)$ such that $Q_0\cap R=P$. Then $dt=\sum_{i=1}^np_it_i\in PT\subseteq Q_0$, although $t\notin Q_0$. Hence $d\in Q_0$, and so $d\in Q_0\cap R=P$, the desired contradiction, to complete the proof.
\end{proof}

Let $R=\bigoplus_{\alpha\in\Gamma}R_{\alpha}$ be a graded integral
domain and $T$ be a homogeneous overring of $R$. As in \cite{hs18a}, we say that $T$ is
an \emph{h-flat} overring of $R$ if for each homogeneous prime ideal
$Q$ of $T$, one has $R_{Q\cap R}=T_Q$. This is a graded analogue of flat overrings investigated by Richman \cite[Theorem 2]{r65}.

\begin{proposition} \label{h-flat}
Suppose that the graded domain $T$ is an h-flat overring of $R$. Then $R\subseteq T$ satisfies $\gGD$.
\end{proposition}
\begin{proof}
Assume that $P_0\subseteq P_1$ are homogeneous prime ideals of $R$ and $Q_1$ is a homogeneous prime ideal of $T$ such that $Q_1\cap R=P_1$. Then $R_{Q_1\cap R}=T_{Q_1}$ is a flat overring of $R$ and $Q_1T_{Q_1}\cap R=P_1$. Thus by \cite{DP2}, there exists a prime ideal $Q_0$ of $T$ such that $Q_0\subseteq Q_1$ and $P_0=Q_0T_{Q_1}\cap R=Q_0\cap R$. Using \cite[Lemma 2.7]{hs18}, we have $P_0=P_0^*=Q_0^*\cap R$, and hence we can assume that $Q_0=Q_0^*$ is homogeneous to complete the proof.
\end{proof}

\section{Graded going down domains}

The most natural examples of going-down domains are Pr\"ufer domains, domains of (Krull) dimension at most $1$, and certain pullbacks. In this section we define the graded-analogue of going-down domains and study their properties.

\begin{definition} Let $R=\bigoplus_{\alpha\in\Gamma}R_{\alpha}$ be a graded domain. Then $R$ is said to be a graded going-down domain (for short, a $\gGD$ domain) if, for every homogeneous overring $T$ of $R$, the extension $R\subseteq T$ satisfies $\gGD$.
\end{definition}

It is easy to show that a $\GD$ graded domain is a $\gGD$ domain. Note that Example \ref{e} shows that the converse is no longer true.

We say that $R= \bigoplus_{\alpha\in\Gamma}R_{\alpha}$ is a {\em graded-valuation domain} (gr-valuation domain) if either $aR \subseteq bR$ or $bR \subseteq aR$ for all $a,b \in H$.

\begin{lemma}\label{vgd} Every gr-valuation domain is a $\gGD$ domain.
\end{lemma}
\begin{proof}
Let $T$ be a homogeneous overring of a gr-valuation domain $V$. Suppose that $P_0\subseteq P$ are homogeneous prime ideals of $V$ and $Q$ is a homogeneous prime ideal of $T$ such that $Q\cap V=P$. Thus $V_P\subseteq T_Q$ and $QT_Q\cap V_P=PV_P$. Note that $V_P$ is a valuation domain \cite[Theorem 4.4]{s14} and hence is a $\GD$ domain. So there exists a prime ideal $Q_0$ of $T$ such that $Q_0\subseteq Q$ and $Q_0T_Q\cap V_P=P_0V_P$. Then $Q_0\cap V=P_0$. It is easy to see that $Q_0=Q_0^*$ is a homogeneous prime ideal of $V$.
\end{proof}

Recall that a commutative ring $A$ is said to be \emph{treed} in case, as a poset under inclusion, $\Spec(A)$ is a tree. As the graded analogue, we next define the notion of a \emph{g-treed domain}. The graded domain $R$ is said to be a {\it g-treed domain} if, as a poset under inclusion, $\Spec_h(R)$ is a tree; i.e., if no  homogeneous prime ideal of $R$ contains incomparable  homogeneous prime ideals of $R$. The proof of the following proposition is a straightforward adaptation of the proof of \cite[Theorem 2.2]{d74}; for the sake of completeness, we include the details. We note that for an extension of $\Gamma$-graded domains $R\subseteq S$ and a homogeneous element $u\in S$ of degree $\alpha$, the algebra $R[u]$ is a $\Gamma$-graded domain in which $(R[u])_{\beta}=\sum_{n\geq0}R_{\beta-n\alpha}\ u^n$ for each $\beta\in\Gamma$.

\begin{proposition}\label{gtreed}
Let $R=\bigoplus_{\alpha\in\Gamma}R_{\alpha}$ be a graded domain. If $R$ is a $\gGD$ domain, then $R$ is a g-treed domain.
\end{proposition}
\begin{proof}
Suppose the assertion fails. Then some homogeneous prime ideal $M$ of $R$ contains incomparable homogeneous prime ideals $P_1$ and $P_2$ of $R$. Choose homogeneous elements $b\in P_1\setminus P_2$ and $c\in P_2\setminus P_1$. Set $u:=bc^{-1}$. By \cite[Theorem 55]{k74}, $MR_M$ survives in either $R_M[u]$ or $R_M[u^{-1}]$, say, in $R_M[u]$. Hence $MR_{H\setminus M}[u]\neq R_{H\setminus M}[u]$. Choose a homogeneous maximal ideal $N$ of $R_{H\setminus M}[u]$ which contains $MR_{H\setminus M}[u]$. Necessarily, $N\cap R_{H\setminus M}=MR_{H\setminus M}$, and so $N\cap R=M$. Since $R$ is a $\gGD$-domain, $R\subseteq R_{H\setminus M}[u]$ satisfies $\gGD$. Thus, there exists a homogeneous prime ideal $Q$ of $R_{H\setminus M}[u]$ contained in $N$ and lying over $P_2$. Then $c\in Q$ and $b=cu\in Q\cap R=P_2$, contradicting the choice of $b$.
\end{proof}

Recall that a domain $D$ is called a \emph{finite conductor domain} if every intersection of two principal ideals of $D$ is finitely
generated. All Pr\"{u}fer damains are finite conductor domains. Let $R=\bigoplus_{\alpha\in\Gamma}R_{\alpha}$ be a graded domain. We define $R$ to be a {\it g-finite conductor} domain if, for all $a$, $b\in H$, there exist finitely many elements $u_1,\ldots,u_n\in H$ such that
$(a)\cap(b)=(u_1,\ldots,u_n)$.

Recall also that a graded domain $R$ is called a graded-Pr\"{u}fer domain if each nonzero finitely generated homogeneous ideal of $R$ is invertible (see \cite{ac13}). It is well-known that $R$ is a graded-Pr\"{u}fer domain if and only if $R_{H\setminus P}$ is a gr-valuation domain for each $P\in\Spec_h(R)$ (\cite[Theorem 3.1]{acz18} or \cite[Theorem 4.4]{s14}).

\begin{lemma}\label{gfc} Let $R=\bigoplus_{\alpha\in\Gamma}R_{\alpha}$ be a graded domain. If $R$ is a graded-Pr\"{u}fer domain, then $R$ is a g-finite conductor domain.
\end{lemma}
\begin{proof}
It follows from the fact that $(a)\cap(b)=ab(a,b)^{-1}$ for all $a,b\in H$, and that $R$ is a graded-Pr\"{u}fer domain.
\end{proof}

In order to prove Theorem \ref{qP}, we need Proposition \ref{conductor}. For this we need the following lemma which is essential \cite[Lemma]{mc72}.

\begin{lemma}\label{mc} Let $R=\bigoplus_{\alpha\in\Gamma}R_{\alpha}$ be a graded domain and let $u\in R_H$. Suppose that $R$ is integrally closed in $R[u]$, and $I=(R:_Ru)$ is a finitely generated ideal of $R$. If $Iu\subseteq\sqrt{I}$, then $u\in R$.
\end{lemma}

To prove the next proposition, we may use the same argument as employed in the proof of \cite[Theorem 1]{mc72}. One might take into consideration that in our proof we appeal to Lemma \ref{mc} instead of \cite[Lemma]{mc72} in the context the proof of \cite[Theorem 1]{mc72}.

\begin{proposition}\label{conductor} Let $R=\bigoplus_{\alpha\in\Gamma}R_{\alpha}$ be an integrally closed graded domain whose homogeneous prime ideals are linearly ordered. Assume that $R$ is a g-finite conductor domain. Then $R$ is a gr-valuation domain.
\end{proposition}

\begin{theorem}\label{qP} Let $R=\bigoplus_{\alpha\in\Gamma}R_{\alpha}$ be a graded domain. Then $R$ is a graded-Pr\"{u}fer domain if and
only if the following three conditions hold:
\begin{itemize}
\item[(1)] $R$ is an integrally closed domain,

\item[(2)] $R$ is a g-finite conductor domain, and

\item[(3)] $R$ is a g-treed domain.
\end{itemize}
\end{theorem}
\begin{proof}
It is enough to show by \cite[Theorem 4.4]{s14} that $R_{H\setminus P}$ is a gr-valuation domain for each homogeneous prime ideal $P$ of $R$. Let $P$ be  a homogeneous prime ideal of $R$. Since $R$ is integrally closed, $R_{H\setminus P}$ is an integrally closed domain. To show that $R_{H\setminus P}$ is g-finite conductor, let $a,b\in H$. Then there are $u_1,\ldots,u_n\in H$ such that $(a)\cap(b)=(u_1,\ldots,u_n)$. Then $(a/1)\cap(b/1)=((a)\cap(b))R_{H\setminus P}=(u_1,\ldots,u_n)R_{H\setminus P}$. This shows that $R_{H\setminus P}$ is a g-finite conductor domain. Finally the homogeneous prime ideals of $R_{H\setminus P}$ are linearly ordered by inclusion since $R$ is g-treed. Thus $R_{H\setminus P}$ is a gr-valuation domain by Proposition \ref{conductor}.
\end{proof}

\begin{corollary}\label{3.11} Let $R=\bigoplus_{\alpha\in\Gamma}R_{\alpha}$ be a graded domain. Then $R$ is a graded-Pr\"{u}fer domain if and
only if the following three conditions hold:
\begin{itemize}
\item[(1)] $R$ is an integrally closed domain,

\item[(2)] $R$ is a g-finite conductor domain, and

\item[(3)] $R$ is a $\gGD$ domain.
\end{itemize}
\end{corollary}
\begin{proof}
In view of Proposition \ref{gtreed}, Lemma \ref{gfc} and Theorem \ref{qP}, it is enough to show that every graded-Pr\"{u}fer domain is a $\gGD$ domain. Assume that $R$ is a graded-Pr\"{u}fer domain and $T$ is an overring of $R$. Then $T$ is an h-flat overring of $R$ by \cite[Lemma 3.4]{hs18a}. Hence $R\subseteq T$ satisfies $\gGD$ by Proposition \ref{h-flat}. Therefore $R$ is a $\gGD$ domain.
\end{proof}

As we saw in the beginning of the present section that every $\GD$ graded domain is a $\gGD$ domain. The next example shows that the converse is no longer true.

\begin{example}\label{e} Let $k$ be a field and $X$ and $Y$ be algebraically independent indeterminates over $k$. Put $R=k[X,X^{-1}][Y]$, which is a $\mathbb{Z}\oplus(\mathbb{N}\cup\{0\})$-graded integral domain. Then $R$ is a 2-dimensional Noetherian domain, and hence it is not a $\GD$ domain by \cite[Proposition 7]{d73}. However it is a $\gGD$ domain, since $R$ is a graded-Pr\"{u}fer domain by \cite[Theorem 6.2]{acz18}.
\end{example}

In \cite{ss23}, we investigated the $\gGD$ property of graded domains in pullback of graded domains and gave various examples of $\gGD$ domains.

In view of Corollary \ref{3.11}, it seems natural to seek to determine test overrings for the ``$\gGD$ domain" property. Theorem \ref{gr-valuation} will present the  graded-theoretic analogue of the test overring result for going-down domains \cite[Theorem 1]{dp76}. We first give the graded analogue of \cite[Proposition]{dp76}.

\begin{proposition} \label{d} Let $R\subseteq T$ be an extension of graded domains, where $T$ is g-treed with a unique homogeneous
maximal ideal. If $R\subseteq R[u]$ satisfies $\gGD$ for each  homogeneous element $u$ in $T$, then $R\subseteq T$ satisfies $\gGD$.
\end{proposition}

\begin{proof}
Suppose that the assertion fails. Then there exist homogeneous prime ideals $P\subseteq P_1$ of $R$ and a homogeneous prime ideal $Q_1$ in $T$ such that $Q_1\cap R=P_1$ and no $Q\in\Spec_h(T)$ satisfies both $Q\subseteq Q_1$ and $Q\cap R=P$. By \cite[Exercise 37, page 44, (ii)$\Rightarrow$ (i)]{k74}, the prime ideal $N$ of $T$ which is minimal over $PT$ also satisfies $P_2:=N\cap R\supset P$. Note that $N$ is homogeneous and uniquely determined since
$T$ has a unique homogeneous maximal ideal and it is g-treed and that $N=\sqrt{PT}$. Thus, choosing a homogeneous element $r\in(N\cap R)\setminus P$ leads to an equation $r^m=\sum p_iw_i$ for some homogeneous elements $p_i$ in $P$ and $w_i$ in $T$ and $m\geq1$.

Considering the radicals of $p_iT$ and using the assumption that the homogenous prime ideals of $T$ are linearly ordered by inclusion, we may relabel the $p_j$ so that, for each $i$, $p_1$ divides a power of $p_i$ (with quotient in $T$). Raising the above equation to a suitably high power, say the $t$th, gives a homogeneous element $w$ in $T$ such that $r^{mt}=p_1w$. Since $PT\subseteq N$, $N\cap R[w]$ is a homogeneous prime ideal of $R[w]$ containing $PR[w]$. Note that $N\cap R[w]$ contracts to $P_2=N\cap R$. Since $R\subseteq R[w]$ satisfies $\gGD$, any prime ideal $N_1$ in $R[w]$ minimal over $PR[w]$ contracts to $P$ \cite[Exercise 37, page 44]{k74}. Note that $N_1$ is a homogeneous prime ideal of $R[w]$ which is contained in $N\cap R[w]$. Hence $r^{mt}=p_1w\in PR[w]\subseteq N_1$, whence $r\in N_1\cap R = P$, a contradiction.
\end{proof}

\begin{lemma}\label{v} Let $R=\bigoplus_{\alpha\in\Gamma}R_{\alpha}$ be a graded domain. Let $W$ be a gr-valuation domain such that $R_H$ is a graded subring of the homogeneous quotient field of $W$, and put $V:=\bigoplus_{\alpha\in\langle\Gamma\rangle}(W\cap(R_H)_{\alpha})$. Then
\begin{itemize}
\item[(1)] $V$ is a gr-valuation overring of $R$;
\item[(2)] Each homogeneous ideal of $V$ is the contraction of a homogeneous ideal of $W$;
\item[(3)] $V\subseteq W$ satisfies $\gGD$.
\end{itemize}
\end{lemma}
\begin{proof}
First of all we note that $V$ is a graded subring of $W$ and a homogeneous overring of $R$. Indeed, for each $\alpha\in\langle\Gamma\rangle$, we have $V_{\alpha}=W\cap(R_H)_{\alpha}\subseteq W\cap (W_{H(W)})_{\alpha}=W_{\alpha}$ and, for each $\alpha\in\Gamma$, we have $R_{\alpha}\subseteq W\cap(R_H)_{\alpha}=V_{\alpha}$.

(1) Assume that $x$ is a nonzero homogeneous element of $R_H$. Then $x$ or $x^{-1}$ is in $W$, and hence $x$ or $x^{-1}$ is in $V$. Thus $V$ is a gr-valuation domain.

(2) Let $I$ be a homogeneous ideal of $V$. Since $V$ is a graded subring of $W$, $IW\cap V$ is a homogeneous ideal of $V$. It is enough to show that $IW\cap V\subseteq I$. To this, assume that $x\in IW\cap V$ is homogeneous. Then $x=\sum_{i=1}^na_ix_i$ for some finite homogeneous subsets $\{a_i\}_{i=1}^n$ and $\{x_i\}_{i=1}^n$ of $I$ and $W$, respectively. The ideal of $V$ generated by $\{a_i\}_{i=1}^n$ is principal and is generated by a homogeneous element $a$ of $I$. We therefore have  $x=\sum_{i=1}^na_ix_i\in aW$. Hence $x/a$ is a homogeneous element of $W$ of degree, say, $\alpha$. So that $x/a\in W\cap (R_H)_{\alpha}=V_{\alpha}\subseteq V$. It follows that $x\in aV\subseteq I$.

(3) Assume that $P_0\subsetneq P_1$ are homogeneous prime ideals of $V$ and $Q_1$ is a homogeneous prime ideal of $W$ such that $Q_1\cap V=P_1$. By (2), there exists a homogeneous ideal $Q_0$ of $W$ such that $Q_0\cap V=P_0$. Since $\sqrt{Q_0}\cap V=P_0$, we can assume that $Q_0=\sqrt{Q_0}$ is a prime ideal of $W$. Now we have $Q_0\subseteq Q_1$ or $Q_1\subseteq Q_0$. But if $Q_1\subseteq Q_0$, then $P_1\subseteq P_0$, which is a contradiction. Thus we have $Q_0\subseteq Q_1$.  \end{proof}

\begin{theorem}\label{gr-valuation} Let $R=\bigoplus_{\alpha\in\Gamma}R_{\alpha}$ be a graded domain. Then the following are equivalent:
\begin{itemize}
\item[(1)] $R$ is a $\gGD$ domain;
\item[(2)] $R\subseteq R[u]$ satisfies $\gGD$ for each homogeneous element $u\in R_H$;
\item[(3)] $R\subseteq V$ satisfies $\gGD$ for each gr-valuation overring $V$;
\item[(4)] $R\subseteq T$ satisfies $\gGD$ for each graded domain $T$ containing $R$ as a graded subring.
\end{itemize}
\end{theorem}
\begin{proof}
(1)$\Rightarrow$(2) and (4)$\Rightarrow$(3) are trivial.

(3)$\Rightarrow$(1) Let $P_0\subseteq P_1$ be homogeneous prime ideals of $R$ and $Q_1$ be a homogeneous prime ideal of a homogeneous overring $T$ of $R$ such that $Q_1\cap R=P_1$. By \cite[Lemma 4.4]{co2018}, $T$ is contained in a gr-valuation domain $V$ whose homogeneous maximal ideal contracts to $Q_1$. Then there exists a homogeneous prime ideal $N$ of $V$ lying over $P_0$, whence $Q_0:=N\cap T$ is contained in $Q_1$ and lies over $P_0$.

(2)$\Rightarrow$(4) Let $T$ be a graded integral domain containing $R$ as a graded subring, let $P\subseteq M$ be homogeneous prime ideals of $R$, and let $N$ be a homogeneous prime ideal of $T$ such that $N\cap R=M$. By \cite[Lemma 4.4]{co2018} $T$ is contained in a gr-valuation domain $W$ whose homogeneous maximal ideal contracts to $N$. Let $V:=\bigoplus_{\alpha\in\langle\Gamma\rangle}(W\cap(R_H)_{\alpha})$. Then $V$ is a gr-valuation domain. By Lemma \ref{v}, $V\subseteq W$ satisfies $\gGD$. Since $V$ has a unique maximal homogeneous ideal and it is g-treed, Proposition \ref{d} implies that $R\subseteq V$ satisfies $\gGD$, and hence $R\subseteq W$ also satisfies $\gGD$. We thus obtain a  homogeneous prime $Q$ of $W$ contracting to $P$. Hence $Q\cap T$ is contained in $N$ and also contracts to $P$.
\end{proof}

The h-height of a homogeneous prime ideal $P$ (denoted by h-ht$(P)$) is defined to be the supremum of the lengths of chains of homogeneous prime ideals descending from $P$ and the h-dimension of $R$ (denoted by h-$\dim(R)$) is defined to be $\sup\{h$-$\mathrm{ht} (P)\mid P$ is a homogeneous prime ideal of $R\}$. Recall that $R$ is called a \emph{gr-Noetherian domain} if the homogeneous ideals of $R$ satisfy the ascending chain condition. The following result is the graded analogue of \cite[Proposition 7]{d73}.

\begin{proposition} \label{} Let $R=\bigoplus_{\alpha\in\Gamma}R_{\alpha}$ be a gr-Noetherian domain. Then the following conditions are equivalent:
\begin{itemize}
\item[(1)] h-$\dim(R)\leq1$;
\item[(2)] If $T$ is any homogeneous overring of $R$, then $T$ is a $\gGD$ domain;
\item[(3)] $R$ is a $\gGD$ domain.
\end{itemize}
\end{proposition}

\begin{proof} $(1)\Rightarrow(2)$: Assume (1). Let $T$ be a homogeneous overring of $R$. Then by \cite[Corollary 1.6]{park}, $T$ is a
gr-Noetherian domain and h-$\dim(T)\leq1$. Hence, $T$ is a $\gGD$ domain.

$(2)\Rightarrow(3)$: This is trivial.

$(3)\Rightarrow(1)$: We will prove the contrapositive. So there exists a saturated chain of homogeneous prime ideals $P_0\subsetneq P_1\subsetneq P_2$ of $R$. By \cite[Lemma 4.4 and Theorem 4.5]{co2018}, there is a gr-valuation overring $V$ of $R$ with a chain $Q_1\subsetneq Q_2$ such that $Q_i\cap R=P_i$, for $i=1,2$ and $V_{H\setminus P_2}$ is a discrete as a gr-valuation ring and that h-$\dim V_{H\setminus P_2}=1$. Then $R\subseteq V_{H\setminus P_2}$ does not satisfy $\gGD$, and so $R$ is not a $\gGD$ domain.
\end{proof}

\section{Graded going-down property is stable under the formation of rings of fractions and factor domains}

This section is devoted to examine the $\gGD$ property of graded domains under localization and factor domains. To this end we need to introduce the notion of graded divided domains and homogeneous unibranched extensions.

We first show that being graded going-down is a local property, see \cite[Lemma 2.1]{d74} and \cite[Proposition 3.8]{ds11}.

\begin{proposition}\label{local}
Let $R=\bigoplus_{\alpha\in\Gamma}R_{\alpha}$ be a graded domain. Then the following conditions are equivalent:
\begin{itemize}
\item[(1)] $R$ is a $\gGD$ domain;
\item[(2)] $R_{H\setminus P}$ is a $\gGD$ domain for all $P\in \Spec_h(R)$;
\item[(3)] $R_{H\setminus M}$ is a $\gGD$ domain for all $M\in \Max_h(R)$.
\end{itemize}
\end{proposition}

\begin{proof} $(1)\Rightarrow(2)$ Let $P$ be a homogeneous prime ideal of $R$. Note that a homogeneous overring of $R_{H\setminus P}$ is a homogeneous overring of $R$ and $\Spec_h(R_{H\setminus P})=\{QR_{H\setminus P}\mid Q\subseteq P,Q\in\Spec_h(R)\}$. Let $T$ be a homogeneous overring of $R_{H\setminus P}$. Suppose that $P_1R_{H\setminus P}\subseteq P_2R_{H\setminus P}$ are homogeneous prime ideals of $R_{H\setminus P}$ and $Q_2$ is a homogeneous prime ideal of $T$ such that $Q_2\cap R_{H\setminus P}=P_2R_{H\setminus P}$. Note that $Q_2\cap R = P_2$. So there is a homogeneous prime ideal $Q_1$ of $T$ such that $Q_1\subseteq Q_2$ and $Q_1\cap R = P_1$. Hence $Q_1\cap R_{H\setminus P}=P_1R_{H\setminus P}$.

$(2)\Rightarrow(3)$ Trivial.

$(3)\Rightarrow(1)$ Let $T$ be a homogeneous overring of $R$. Suppose that $P_1\subseteq P_2$ are homogeneous prime ideals of $R$ and $Q_2$ is a homogeneous prime ideal of $T$ such that $Q_2\cap R=P_2$. We must find a homogeneous prime ideal $Q_1$ of $T$ such that both $Q_1\subseteq Q_2$ and $Q_1\cap R=P_1$. Choose a maximal homogeneous ideal $M$ of $R$ which contains $P_2$. It is enough to find a homogeneous prime ideal $Q_1$ of $T$ such that $Q_1\subseteq Q_2$ and $Q_1\cap R_{H\setminus M}=P_1R_{H\setminus M}$. This can be done thanks to (3), as the ring extension $R_{H\setminus M} \subseteq T_{H\setminus M}$ satisfies $\gGD$.
\end{proof}

\begin{remark}\label{l}
By the same argument as in part $(1)\Rightarrow(2)$ of the above proof, one can show that if $R$ is a $\gGD$ domain and $S\subseteq H$ is a multiplicatively closed subset, then $R_S$ is a $\gGD$ domain. Note that a homogeneous overring of $R_S$ is a homogeneous overring of $R$ and $\Spec_h(R_S)=\{QR_S\mid Q\cap S=\emptyset,Q\in\Spec_h(R)\}$.
\end{remark}

Now we study the $\gGD$ property under the formation of factor domains. For this we need the notion of graded divided integral domains.
Let $R=\bigoplus_{\alpha\in\Gamma}R_{\alpha}$ be a graded integral domain. A homogeneous prime ideal $P$ of $R$ is called a \emph{homogeneous divided prime ideal} if each element of $H\setminus P$ divides each  homogeneous element of $P$. It can be seen easily that $P$ is a homogeneous divided prime ideal if and only if $P$ is comparable to each homogeneous ideal of $R$ if and only if $P=PR_{H\setminus P}$. The graded integral domain $R$ is called a \emph{graded divided integral domain} if each homogeneous prime ideal of $R$ is a homogeneous divided prime ideal. This is the graded analogue of divided rings introduced by Akiba \cite{a67} as AV-domains, and then studied further by McAdam and Dobbs \cite{mc76, d76}. Our first result in this case is the graded analogue of \cite[Proposition 2.1]{d76}.

\begin{lemma}\label{2.1}
Any graded divided integral domain has a unique maximal homogeneous ideal and is a $\gGD$ domain.
\end{lemma}
\begin{proof} Assume that $R=\bigoplus_{\alpha\in\Gamma}R_{\alpha}$ is a graded divided integral domain. Since every homogeneous prime ideal is comparable to each homogeneous ideal, $R$ has a unique maximal homogeneous ideal. Now assume that $R$ is not a $\gGD$. Then there exists a homogeneous overring $T$ of $R$ such that the extension $R\subseteq T$ does not satisfy $\gGD$. Hence there are homogeneous prime ideals $P\subseteq P_1$ of $R$ and a homogeneous prime ideal $Q_1$ in $T$ such that $Q_1\cap R=P_1$ and no $Q\in\Spec_h(T)$ satisfies both $Q\subseteq Q_1$ and $Q\cap R=P$. By \cite[Exercise 37, page 44, (iii)$\Rightarrow$ (i)]{k74}, for a prime ideal $Q$ of $T$ minimal over $PT$, we have $PT\cap(R\setminus P)(T\setminus Q)\neq\emptyset$. Then $\Sigma p_it_i=rt$, for some  $r\in R\setminus P$ and $t\in T\setminus Q$ and elements $p_i$ in $P$ and $t_i$ in $T$. Let $r=r_{{\alpha}_1}+\cdots+r_{{\alpha}_n}$ be the decomposition of $r$ into homogeneous components of degrees ${\alpha}_1,\ldots,{\alpha}_n$ with ${\alpha}_1<\cdots<{\alpha}_n$. Since $\Sigma p_it_i-\Sigma_{r_{{\alpha}_j}\in P}r_{{\alpha}_j}t=\Sigma_{r_{{\alpha}_j}\notin P}r_{{\alpha}_j}t$, we can assume that $r_{{\alpha}_j}\in H\setminus P$ for $j=1,\ldots,n$. By decomposing $p_i$ and $t_i$ into homogeneous elements, we can assume that $p_i$ and $t_i$ are homogeneous elements. Now let $t=t_{{\beta}_1}+\cdots+t_{{\beta}_s}+\cdots+t_{{\beta}_m}$ be the decomposition of $t$ into homogeneous components of degrees ${\beta}_1,\ldots,{\beta}_m$ with ${\beta}_1<\cdots<{\beta}_m$, such that ${\beta}_s$ is the smallest degree such that $t_{{\beta}_s}\in T\setminus Q$. The right hand side of the equality
\begin{equation}\label{(0)}
    \Sigma p_it_i=(r_{{\alpha}_1}+\cdots+r_{{\alpha}_n})(t_{{\beta}_1}+\cdots+t_{{\beta}_s}+\cdots+t_{{\beta}_m})
\end{equation}
is the sum of the entries of the following array:
$$
\begin{array}{ccccccc}
  r_{{\alpha}_1}t_{{\beta}_1} & r_{{\alpha}_1}t_{{\beta}_2} & r_{{\alpha}_1}t_{{\beta}_3} & \cdots & r_{{\alpha}_1}t_{{\beta}_s} & \cdots & r_{{\alpha}_1}t_{{\beta}_m} \\
  r_{{\alpha}_2}t_{{\beta}_1} & r_{{\alpha}_2}t_{{\beta}_2} & r_{{\alpha}_2}t_{{\beta}_3} & \cdots & r_{{\alpha}_2}t_{{\beta}_s} & \cdots & r_{{\alpha}_2}t_{{\beta}_m} \\
  r_{{\alpha}_3}t_{{\beta}_1} & r_{{\alpha}_3}t_{{\beta}_2} & r_{{\alpha}_3}t_{{\beta}_3} & \cdots & r_{{\alpha}_3}t_{{\beta}_s} & \cdots & r_{{\alpha}_3}t_{{\beta}_m} \\
  \vdots & \vdots & \vdots &  & \vdots &  & \vdots \\
  r_{{\alpha}_n}t_{{\beta}_1} & r_{{\alpha}_n}t_{{\beta}_2} & r_{{\alpha}_n}t_{{\beta}_3} & \cdots & r_{{\alpha}_n}t_{{\beta}_s} & \cdots & r_{{\alpha}_n}t_{{\beta}_m}.
\end{array}
$$
The component $r_{{\alpha}_1}t_{{\beta}_1}$ has the smallest degree $\alpha_1+\beta_1$ in both sides of the equality (\ref{(0)}).
By canceling   $\Sigma_{\deg(p_it_i)=\alpha_1+\beta_1} p_it_i=r_{{\alpha}_1}t_{{\beta}_1}$ from both sides of the equality (\ref{(0)}), we obtain the equality
\begin{equation}\label{(1)}
    \Sigma p_it_i=\Sigma_{j+k\geq3}r_{{\alpha}_j}t_{{\beta}_k}.
\end{equation}
Now the component of smallest degree in the right side of the equation (\ref{(1)}) is $r_{{\alpha}_1}t_{{\beta}_2}$ or $r_{{\alpha}_2}t_{{\beta}_1}$ or $r_{{\alpha}_1}t_{{\beta}_2}+r_{{\alpha}_2}t_{{\beta}_1}$. In any case we can cancel these components from both sides of the equation (\ref{(1)}) to obtain the equation
\begin{equation*}\label{(2)}
   \Sigma p_it_i=\Sigma_{j+k\geq4}r_{{\alpha}_j}t_{{\beta}_k}.
\end{equation*}
Continuing this procedure we obtain an equality like
\begin{equation}\label{(3)}
    \Sigma p_it_i=\Sigma_{j+k\geq s+1}r_{{\alpha}_j}t_{{\beta}_k}.
\end{equation}
Set $X:=\{r_{{\alpha}_j}t_{{\beta}_k}\mid \alpha_j+\beta_k=\alpha_1+\beta_s$ and $k\neq s\}$. Note that if $r_{{\alpha}_j}t_{{\beta}_k}\in X$, then $j\geq2$ and $k\leq s-1$. This means that $X\subseteq Q$. The component of smallest degree in the right side of the equation (\ref{(3)}) is one of $r_{{\alpha}_j}t_{{\beta}_k}$ with $\alpha_j+\beta_k=\alpha_1+\beta_s$ or a sum of some of these elements. If this component of smallest degree is of the form $\Sigma_{x\in A}x$ for some $A\subseteq X$, then we can cancel this component from both sides of the equation (\ref{(3)}). After that if $\Sigma_{x\in B}x$ for some $B\subseteq X$ has the smallest degree, delete this component too and so on. Hence we can assume that
\begin{equation}\label{(4)}
    r_{{\alpha}_1}t_{{\beta}_s}+\Sigma_{x\in C}x=\Sigma p_it_i
\end{equation}
has the smallest degree for some $C\subseteq X$. Since $PT\subseteq Q$, we have $r_{{\alpha}_1}t_{{\beta}_s}\in Q$. But $t_{{\beta}_s}\notin Q$, and so that $r_{{\alpha}_1}\in Q$. Since $Q$ is a minimal prime ideal of $PT$, using \cite[Theorem 2.1]{Huck} there exists an element $y\in T\setminus Q$, and a nonnegative integer $n$ such that $r_{{\alpha}_1}^ny\in PT$. Note that we can assume that $y$ is a homogeneous element. Indeed let $y=y_{{\gamma}_1}+\cdots+y_{{\gamma}_p}+\cdots+y_{{\gamma}_{\ell}}$ be the decomposition of $y$ into homogeneous elements and assume that $y_{{\gamma}_p}\notin Q$. Thus $r_{{\alpha}_1}^ny_{{\gamma}_1}+\cdots+r_{{\alpha}_1}^ny_{{\gamma}_p}+\cdots+r_{{\alpha}_1}^ny_{{\gamma}_{\ell}}=r_{{\alpha}_1}^ny\in PT$, and so that $r_{{\alpha}_1}^ny_{{\gamma}_p}\in PT$. Hence we can assume that $y$ is homogeneous. Thus $t_{{\beta}_s}y\in T\setminus Q$ is a homogeneous element. Multiplying both sides of (\ref{(4)}) by $r_{{\alpha}_1}^{n}y$, we obtain
$$r_{{\alpha}_1}^{n+1}(t_{{\beta}_s}y)+r_{{\alpha}_1}^ny(\Sigma_{x\in C}x)=\Sigma r_{{\alpha}_1}^np_it_iy.$$ This means that $r_{{\alpha}_1}^{n+1}(t_{{\beta}_s}y)\in PT$. Whence $r_{{\alpha}_1}^{n+1}(t_{{\beta}_s}y)=\Sigma p_it_i$ for some homogeneous elements $p_i\in P$ and $t_i\in T$. Now since $R$ is graded divided, $r_i:=p_i(r_{{\alpha}_1}^{n+1})^{-1}$ is in $P$ and so $t_{{\beta}_s}y=\Sigma r_it_i\in PT\subseteq Q$, which is a contradiction.
\end{proof}

We say that an extension of graded domains $A\subseteq B$ is \emph{homogeneous unibranched} (resp. \emph{satisfies graded lying over} ($\gLO$)) in case the contraction map $\Spec_h(B)\to\Spec_h(A)$ is bijection (resp. surjection). The non-graded version of the next remark is stated in \cite[Lemma 1]{mc72a}.

\begin{remark}\label{mcadam}
(1) If $R \subseteq S\subseteq T$ are extensions of graded domains and $\gGD$ holds in $R\subseteq S$ and $S\subseteq T$, then it holds in $R\subseteq T$.

(2) If $\gGD$ holds in $R\subseteq T$ and $\gLO$ holds in $S\subseteq T$, then $\gGD$ holds in $R\subseteq S$.
\end{remark}

The next two lemmas and theorem are the graded version of \cite[Lemmas 2.3, 2.4 and Theorem 2.5]{d76} respectively.

\begin{lemma}\label{2.3}
Let $R=\bigoplus_{\alpha\in\Gamma}R_{\alpha}$ be a graded subring of $T=\bigoplus_{\beta\in\Lambda}T_{\beta}$ such that $R\subseteq T$ is an integral and homogeneous unibranched extension of graded domains. Then $R$ is a $\gGD$ if and only if $T$ is a $\gGD$.
\end{lemma}
\begin{proof} Since $R\subseteq T$ is integral, it satisfies lying over and going-up properties by \cite[Theorems 5.10 and 5.11]{AM}. Let $P_1\subseteq P_2$ be homogeneous prime ideals of $R$ and $Q_2$ is a homogeneous prime ideal of $T$ such that $Q_2\cap R=P_2$. By lying over, one can choose a prime ideal $Q_1$ of $T$ such that $Q_1\cap R=P_1$. Note that we have $P_1=P_1^*=(Q_1\cap R)^*=Q_1^*\cap R$ by \cite[Lemma 2.7]{hs18}; so that we can assume that $Q_1$ is homogeneous. Hence, by going-up one can choose a prime ideal $Q_2'$ containing $Q_1$ such that $Q_2'\cap R=P_2$. Again, using \cite[Lemma 2.7]{hs18}, we can assume that $Q_2'$ is homogeneous. Now $Q_2=Q_2'$, since the extension is unibranched. Thus $R\subseteq T$ satisfies $\gGD$.

Now assume that $T$ is a $\gGD$ domain. We show that $R\subseteq V$ satisfies $\gGD$ for each gr-valuation overring $V$ of $R$. Let $V$ be a gr-valuation overring of $R$. Observe that $VT$ is a homogeneous overring of $T$ graded by $(VT)_{\gamma}=\sum_{\gamma=\alpha+\beta}V_{\alpha}T_{\beta}$ for $\gamma\in\langle\Lambda\rangle$. Then $T\subseteq VT$ satisfies $\gGD$, since $T$ is a $\gGD$ domain, so that $R\subseteq VT$ satisfies $\gGD$. However, $V\subseteq VT$ satisfies lying over since $VT$ is integral over $V$. An application of Remark \ref{mcadam}(2) implies that $R\subseteq V$ satisfies $\gGD$, as required. Thus $R$ is a $\gGD$ domain by Theorem \ref{gr-valuation}.

Conversely, assume that $R$ is a $\gGD$ domain. It suffices to show that $T\subseteq V$ satisfies $\gGD$ for each gr-valuation overring $V$ of $T$. Let $V$ be a gr-valuation overring of $T$, and $P_1\subseteq P_2$ be homogeneous primes of $T$ and $Q_2$ a homogeneous prime of $V$ contracting to $P_2$. By the hypotheses $R$ is a $\gGD$ domain, and hence $R\subseteq V$ satisfies $\gGD$ property by Theorem \ref{gr-valuation} part $(1)\Leftrightarrow(4)$. Thus one can choose a homogeneous prime  ideal $Q_1$ of $V$ such that $Q_1 \subseteq Q_2$ and $Q_1\cap R=P_1\cap R$. Whence $Q_1\cap T$ and $P_1$ have the same contraction in $R$, and so are equal.
\end{proof}

Let $R=\bigoplus_{\alpha\in\Gamma}R_{\alpha}$ be a graded domain and $P$ be a homogeneous prime ideal of $R$. Then $T=R+PR_{H\setminus P}$ is a $\Gamma$-graded integral domain (see \cite[Page 629]{cs18}).

\begin{lemma}\label{2.4}
Let $R=\bigoplus_{\alpha\in\Gamma}R_{\alpha}$ be a graded domain with unique homogeneous maximal ideal and $\gGD$, let $P$ be a homogeneous prime ideal of $R$ and set $T=R+PR_{H\setminus P}$. Then
\begin{enumerate}
  \item $R\subseteq T$ is integral and homogeneous unibranched.
  \item Let $Q$ be a homogeneous divided prime ideal in $R$ and $Q\subseteq P$. Then $Q$ is a prime of $T$ which contracts to $Q$ and $Q$ is homogeneous divided in $T$.
  \item $PR_{H\setminus P}$ is homogeneous divided in $T$.
\end{enumerate}
\end{lemma}
\begin{proof} (1) For the integrality assertion, it is enough to prove that any nonzero homogeneous $v$ in $PR_{H\setminus P}$ is integral over $R$. Express $v$ as $a/b$ with $a\in P\cap H$ and $b\in H\setminus P$. If $P$ survives in $R[v^{-1}]$, the fact that $R\subseteq R[v^{-1}]$ satisfies $\gGD$, supplies a homogeneous prime $L$ of $R[v^{-1}]$ which contracts to $P$; as $b = av^{-1}$ is then in $PR[v^{-1}]$ it follows that $b$ is in $L\cap R = P$, a contradiction. Thus, 1 is in $PR[v^{-1}]$ that is, $1=p_0+p_1v^{-1}+\cdots+p_nv^{-n}$ for some homogeneous elements $p_i\in P$. Multiplication by $v^n(1-p_0)^{-1}$ produces the desired integrality equation of $v$ over $R$. Note that $1-p_0$ is invertible in $R$.

For the second assertion, let $Q$ be a homogeneous prime ideal of $R$. Since $R$ has a unique homogeneous maximal ideal and is g-treed, either $Q\subsetneq P$ or $P\subseteq Q$. If $Q\subsetneq P$, $QR_{H\setminus P}$  is a homogeneous prime ideal of $T$ contracting to $Q$. Let $I$ be a homogeneous prime ideal of $T$ such that $I\cap R=Q$. Hence $QT=(I\cap R)T\subseteq I$. This shows that $(QR_{H\setminus P})(PR_{H\setminus P})\subseteq I$. It follows from $PR_{H\setminus P}\nsubseteq I$ that $QR_{H\setminus P}\subseteq I$. Now by incomparability $I=QR_{H\setminus P}$. For the case $P\subseteq Q$, note similarly that any prime of $T$ which contracts to $Q$ must contain $(PR_{H\setminus P})^2$ and, hence, coincides with $Q + PR_{H\setminus P}$.

(2) In (1) we saw that $QR_{H\setminus P}$  is a homogeneous prime ideal of $T$ contracting to $Q$. As $Q$ is homogeneous divided in $R$, $Q=QR_{H\setminus P}$ is a prime ideal of $T$. Now it can be seen that $QT_{H'\setminus Q}\subseteq QR_{H\setminus Q}=Q$, where $H'$ is the set of nonzero homogeneous elements of $T$. Hence $Q$ is homogeneous divided in $T$ (see \cite[Lemma 2.4]{d76}).

(3) It is a straightforward computation.
\end{proof}

We are now ready to prove the main result of this section which provides a characterization of $\gGD$ domains in terms of graded divided integral unibranched homogeneous overrings. This leads us to show that the $\gGD$ property goes to graded factor domains.

\begin{theorem}\label{2.5}
Let $R=\bigoplus_{\alpha\in\Gamma}R_{\alpha}$ be a graded domain with a unique homogeneous maximal ideal. The following are equivalent:
\begin{enumerate}
  \item $R$ is a $\gGD$ domain.
  \item $R$ has a graded divided integral unibranched homogeneous overring.
  \item $R$ has a graded divided integral unibranched homogeneous extension.
\end{enumerate}
\end{theorem}
\begin{proof} (1)$\Rightarrow$(2) Let $X$ be the set of homogeneous unibranched integral overrings of $R$. Then $X\neq\emptyset$. For $T_1, T_2\in X$, define $T_1\leq T_2$, whenever $T_1$ is a graded subring of $T_2$. Hence $(X,\leq)$ is a partially ordered set. The union of any chain in $X$ is integral over $R$ and homogeneous unibranched. Indeed, assume that $\{T_i\}$ is a chain in $X$ and set $T:=\bigcup_i T_i$. Then $T=\bigoplus_{\alpha\in\langle\Gamma\rangle}T_{\alpha}$, where $T_{\alpha}=\bigcup(T_i)_{\alpha}$ for $\alpha\in\langle\Gamma\rangle$. Let $P$ be a homogeneous prime ideal of $R$. Then there exists a homogeneous prime ideal $Q_i$ of $T_i$ such that $Q_i\cap R=P$, for each $i$. Assume that $T_i\subseteq T_j$ for some $i$ and $j$. Hence $(Q_j\cap T_i)\cap R=P=Q_i\cap R$ implies that $Q_i=Q_j\cap T_i\subseteq Q_j$. This shows that $Q:=\cup_iQ_i$ is a homogeneous prime ideal of $T$ and $Q\cap R=P$. Assume now that $Q,Q'$ are homogeneous prime ideal of $T$ such that $Q\cap R=Q'\cap R$. Thus for each $i$, $(Q\cap T_i)\cap R=(Q'\cap T_i)\cap R$, hence $Q\cap T_i=Q'\cap T_i$. It follows that $Q=Q'$. Thus by Zorn's Lemma $X$ has a maximal element $T$. We claim that $T$ is a graded divided domain. If $Q$ is not a homogeneous divided prime ideal in $T$, then Lemma \ref{2.4} shows that $T+QT_{H'\setminus Q}$ is integral
and homogeneous unibranched over $T$, and hence over $R$, contradicting maximality in $X$, where $H'$ is the set of nonzero homogeneous elements of $T$. Therefore, $T$ is graded divided, as desired.

(2)$\Rightarrow$(3) It is trivial. (3)$\Rightarrow$(1) It follows from Lemmas \ref{2.1} and \ref{2.3}.
\end{proof}

Let $R =\bigoplus_{\alpha \in \Gamma}R_{\alpha}$ be a $\Gamma$-graded integral domain, and $I = \bigoplus_{\alpha \in \Gamma}I_{\alpha}$ be a homogeneous ideal of $R$. Then $R/I$ is a $\Gamma$-graded ring with $(R/I)_{\alpha}=R_{\alpha}/I_{\alpha}$ for each $\alpha \in \Gamma$ (by $( x_{\alpha_1} + \cdots +x_{\alpha_n}) +I = (x_{\alpha_1} +I_{\alpha_1}) + \cdots +  (x_{\alpha_n} + I_{\alpha_n})$ for $x_{\alpha_i} \in R_{\alpha_i}$). Assume that $R$ is a graded subring of $T =\bigoplus_{\alpha \in \Lambda}T_{\alpha}$, and $Q$ is a homogeneous prime ideal of $T$. Then $R/(Q\cap R)$ is a graded subring of $T/Q$, that is $r_{\alpha}+(Q\cap R)_{\alpha}\mapsto r_{\alpha}+Q_{\alpha}$ for $r_{\alpha}\in R_{\alpha}$ and $\alpha\in\Gamma$.

\begin{corollary}\label{R/P}
(see \cite[Remark 2.11]{d76}) Assume that $R=\bigoplus_{\alpha\in\Gamma}R_{\alpha}$ is a $\gGD$ domain and $P$ is a homogeneous prime ideal of $R$. Then $R/P$ is a $\gGD$ domain.
\end{corollary}
\begin{proof} By Proposition \ref{local} we can assume that $R$ has a unique homogeneous maximal ideal. Then by Theorem \ref{2.5}, $R$ has an integral graded divided and homogeneous unibranched extension $T$. Let $Q$ be a homogeneous prime ideal of $T$ which contracts to $P$. The extension of graded domains $R/P\subseteq T/Q$ is integral and homogeneous unibranched and it can be seen easily that $T/Q$ is graded divided. Now $R/P$ is a $\gGD$ domain by Theorem \ref{2.5}.
\end{proof}

\vspace{.2cm}
\noindent {\bf Acknowledgement.}
The authors would like to thank the referee for his/her careful reading of the manuscript and several wonderful comments which greately improved the paper.

\end{document}